\documentclass[aop,MSNbibl,number,citesort,dvips]{arximspdf}
\usepackage{graphicx}
\doi{10.1214/11-AOP704}
\volume{41}
\issue{2}
\pubyear{2013}
\firstpage{503}
\lastpage{526}

\makeatletter
\newcommand{\rr}{\mathbb{R}}
\newcommand{\abs}[1]{| #1 |}
\newcommand{\usimp}{\mathbb{S}}
\newcommand{\tsigma}{\tilde{\sigma}}
\newcommand{\vZ}{\mathbf{Z}}
\newcommand{\vk}{\mathbf{k}}
\newcommand{\vy}{\mathbf{y}}
\newcommand{\vtheta}{\boldsymbol{\Theta}}
\newcommand{\vone}{\mathbf{1}}
\newcommand{\diri}{\operatorname{Dir}}
\newcommand{\leb}{\mathbf{Leb}}

\newcommand{\tw}{\widetilde{W}}
\newcommand{\ent}{\mathrm{Ent}}
\newcommand{\treesp}{\mathbf{T}}
\newcommand{\tree}{\mathbf{t}}
\newcommand{\ctree}{\mathbb{T}}

\newcommand{\vtwo}{\mathbf{2}}

\newcommand{\ymini}{\vec{y}_i}
\newcommand{\zmini}{\vec{z}_i}
\newcommand{\tmini}{\vec{\theta}_i}

\newcommand{\eqref}[1]{(\ref{#1})}
\newtheorem{lemma}[theorem]{Lemma}
\newtheorem{theorem}{Theorem}
\newproclaim{rmk}{Remark}
\renewcommand{\epsilon}{\varepsilon}
\makeatother

\begin{document}
\begin{frontmatter}

\title{Wright--Fisher diffusion with negative mutation~rates}
\runtitle{Negative Wright--Fisher}

\begin{aug}
\author[A]{\fnms{Soumik} \snm{Pal}\corref{}\ead[label=e1]{soumik@u.washington.edu}}
\runauthor{S. Pal}
\affiliation{University of Washington}
\address[A]{Department of Mathematics\\
University of Washington\\
Seattle, Washington 98115\\
USA\\
\printead{e1}} 
\end{aug}

\received{\smonth{2} \syear{2010}}
\revised{\smonth{8} \syear{2011}}

%
\begin{abstract}
We study a family of $n$-dimensional diffusions, taking values in the
unit simplex of vectors with nonnegative coordinates that add up to
one. These processes satisfy stochastic differential equations which
are similar to the ones for the classical Wright--Fisher diffusions,
except that the ``mutation rates'' are now nonpositive. This model,
suggested by Aldous, appears in the study of a conjectured diffusion
limit for a Markov chain on Cladograms. The striking feature of these
models is that the boundary is not reflecting, and we kill the process
once it hits the boundary. We derive the explicit exit distribution
from the simplex and probabilistic bounds on the exit time. We also
prove that these processes can be viewed as a ``stochastic
time-reversal'' of a Wright--Fisher process of increasing dimensions
and conditioned at a random time. A key idea in our proofs is a
skew-product construction using certain one-dimensional diffusions
called Bessel-square processes of negative dimensions, which have been
recently introduced by G\"oing-Jaeschke and Yor.
\end{abstract}

%

\end{frontmatter}

\section{Introduction}

An $n$-leaf Cladogram is an unrooted tree with $n\ge4$ labeled leaves
(vertices with degree one) and $(n-2)$ other unlabeled vertices
(internal branchpoints) of degree three (see Figure~\ref{fig_tree}). The
number of edges in such a tree is exactly $2n-3$. Sometimes they are
also referred to as phylogenetic trees. Aldous, in~\cite{A00},
proposes the following model of a reversible Markov chain on the space
of all $n$-leaf Cladograms, which consists of removing a random leaf
(and its incident edge) and reattaching it to one of the remaining
random edges.

For a precise description we first define two operations on Cladograms.
More details, with figures, can be found in~\cite{A00}.
\begin{longlist}[(ii)]
\item[(i)] To \textit{remove a leaf} $i$. The leaf $i$ is attached by
an edge $e_1$ to a branchpoint $b$ where two other edges $e_2$ and
$e_3$ are incident. Delete edge $e_1$ and branchpoint $b$, and then
merge the two remaining edges $e_2$ and $e_3$ into a single edge $e$.
The resulting tree has $2n-5$ edges.
\item[(ii)] To \textit{add a leaf} to an edge $f$. Create a
branchpoint $b'$ which splits the edge $f$ into two edges, $f_2, f_3$,
and attach the leaf $i$ to branchpoint $b'$ via a new edge, $f_1$. This
restores the number of leaves and edges to the tree.
\end{longlist}

Let $\treesp_n$ denote the finite collection of all $n$-leaf
Cladograms. Write $\tree' \sim\tree$ if $\tree'\neq\tree$ and
$\tree'$ can be obtained from $\tree$ by following the two operations
above for some choice of $i$ and $f$. Thus a $\treesp_n$ valued chain
can be described by saying: remove leaf $i$ uniformly at random, and
then pick edge $f$ at random and reattach $i$ to $f$. If we assume
every edge to be of unit length, then it also involves resizing the
edge length after every operation. In particular the transition matrix
of this Markov chain is
\[
P(\tree, \tree')=
\cases{\displaystyle
\frac{1}{n(2n-5)},&\quad  if  $\tree'\sim\tree$,\vspace*{2pt}\cr\displaystyle
\frac{n}{n(2n-5)},&\quad  if  $\tree'=\tree$.
}
\]
This leads to a symmetric, aperiodic, and irreducible finite state
space Markov chain. Schweinsberg~\cite{S} proved that the relaxation
time for this chain is $O(n^2)$, improving a previous result in~\cite{A00}.

On his webpage~\cite{AOP} Aldous asks the following question: what is
an appropriate diffusion limit of this Markov chain? The invariant
distribution for the Markov chain on $n$-leaf Cladograms is clearly the
Uniform distribution. It is known (see Aldous~\cite{A93}) that the
sequence of Uniform distributions on $n$-leaf Cladograms converge
weakly to the law of the (Brownian) Continuum Random Tree (CRT). Hence,
it is natural to look for an appropriate Markov process on the support
of the CRT, which can be thought of as a limit of the sequence of
Markov chains described above. At this point it is important to
understand that the support of the CRT consists of compact real trees
with a measure describing the distribution of leaves. These trees are
called continuum trees. For a formal definition of these concepts, we
refer the reader to the seminal work by Aldous in~\cite{A93}. However,
for an intuitive visualization, one should think of a typical continuum
tree as a compact metric space on which branch points are dense, and
all edges are infinitesimally small. This implies that the Markov
process that mimics the operation of removing and inserting a new leaf
on a continuum tree should not jump; in other words, we can call it a diffusion.

A detailed description of this diffusion on continuum trees is
forthcoming in Pal~\cite{palCT}. In this article we consider several
important features of this limiting diffusion that are of interest by
themselves and provide bedrock for the followup construction.

%
\begin{figure}

\includegraphics{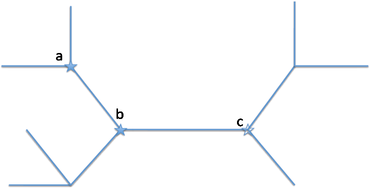}

\caption{A 7-leaf Cladogram.}\label{fig_tree}
\label{fig1}\end{figure}

Consider the branchpoint $b$ in the $7$-leaf Cladogram $\tree$ in
Figure~\ref{fig1}. It divides the collection of leaves naturally into three
sets. Let $X(\tree)=(X_1, X_2,\break X_3)(\tree)$ denote the vector of
proportion of leaves in each set. The corresponding number of edges in
these sets are $(2nX_1-1, 2nX_2-1, 2nX_3-1)$. For example, at time zero
in our given tree, going clockwise from the right we have $X(0)=(3/7,
2/7, 2/7)$.

Let $\usimp_n$ denote the unit simplex
%
\begin{equation}\label{whatisusimp}
\usimp_n = \Biggl\{ x\in\rr^n\dvtx  x_i \ge0 \mbox{ for all
$i$ and } \sum_{i=1}^n x_i=1 \Biggr\}.
\end{equation}
Some simple algebra will reveal that for any point $x=(x_1,x_2,x_3)$ in
$\usimp_3$, given $X(\tree)=x$, the difference $X_1(\tree')-
X_1(\tree)$ can only take values in $\{ -1/n,0,1/n\}$ with
corresponding probabilities
\[
q_{x_1}= x_1 \frac{2n(1-x_1)-2}{2n-5}, \qquad1-p_{x_1}-q_{x_1}, \qquad
p_{x_1}=(1-x_1)\frac{2n x_1-1}{2n-5}.
\]
Thus
%
\begin{eqnarray}\label{cladopara}
 \quad E \bigl( X_1(\tree') - X_1(\tree) \mid X(\tree)=x \bigr)&=& \frac
{1}{n}\frac{ 2x_1 - (1-x_1)}{2n-5}\approx-\frac{1}{n^2}\frac
{1}{2} ( 1-3x_1 ),\nonumber\\
 \quad E \bigl( \bigl( X_1(\tree') - X_1(\tree)\bigr)^2 \mid X(\tree)=x \bigr)&=&
\frac{1}{n^2}\frac{4nx_1(1-x_1)-x_1 - 1 }{2n-5}\\
 \quad &\approx&\frac{1}{n^2} 2x_1(1-x_1).\nonumber
\end{eqnarray}
If we take scaled limits, as $n$ goes to infinity, of the first two
conditional moments (the mixed moments can be similarly verified), it
is intuitive (and follows by standard tools) that as $n$ goes to
infinity, this Markov chain (run at $n^2/2$ speed) will converge to a
diffusion with a generator
%
\begin{equation}\label{cladogen}
\frac{1}{2}\sum_{i,j=1}^3 x_i ( 1\{i=j\} - x_j ) \frac
{\partial^2}{\partial x_i\, \partial x_j} - \frac{1}{2}\sum_{i=1}^n
\frac{1}{2} (1 - 3 x_i )\frac{\partial}{\partial x_i}.
\end{equation}

The generator written as above is similar to the generator for the
well-known diffusion limit of the Wright--Fisher (WF) Markov chain
models in population genetics. The WF model is one of the most popular
models in population genetics. This is a multidimensional Markov chain
which keeps track of the vector of proportions of certain genetic
traits in a population of nonoverlapping generations. A~good source
for an introduction to these models is Chapter 1 in the book by Durrett
\cite{durrettgenetics}. For computational purposes one often takes
recourse to a diffusion approximation, which, in its standard form,
leads to a family of diffusions parametrized by $n$ ``mutation rates.''
The state space of the diffusion is given by $\usimp_n$ and is
parametrized by a vector $(\delta_1, \ldots, \delta_n)$ of
nonnegative entries. A weak solution of the WF diffusion with
parameters $\delta=(\delta_1, \ldots, \delta_n)$ solves the
following stochastic differential equation for $i=1,2,\ldots,n$:
%
\begin{equation}\label{whatisjacobi}
dJ_i(t) = \frac{1}{2} \bigl(\delta_i - \delta_0 J_i(t) \bigr)\,dt +
\sum_{j=1}^n \tsigma_{i,j}(J)\,d\beta_j(t), \qquad\delta_0 = \sum
_{i=1}^n \delta_i.
\end{equation}
Here $\beta=(\beta_1, \ldots, \beta_n)$ is a standard
multidimensional Brownian motion, and the diffusion matrix $\tsigma$
is given by
%
\begin{equation}\label{whatistsigma}
\tsigma_{i,j}(x)= \sqrt{x_i} \bigl(1\{i=j\} - \sqrt{x_ix_j}
\bigr), \qquad1\le i,j\le n.
\end{equation}

We define the Wright--Fisher diffusion with \textit{negative mutation
rates} to be a family of $n$-dimensional diffusions, parametrized by
$n$ nonnegative parameters $\delta=(\delta_1, \ldots, \delta_n)$,
which is a weak solution of the following differential equation:
%
\begin{equation}\label{whatisnwf}
d\mu_i(t) = -\frac{1}{2} \bigl(\delta_i - \delta_0 \mu_i(t)
\bigr)\,dt + \sum_{j=1}^n \tsigma_{i,j}(\mu)\,d\beta_j(t), \qquad\delta
_0 = \sum_{i=1}^n \delta_i.
\end{equation}
%

The initial condition $\mu(0)$ is in the interior of $\usimp_n$ and
the process has a drift that pushes it outside the simplex. We will
show later that the process is sure to hit the boundary of the simplex
at which point we stop it. In the next section we will explicitly
construct a weak solution of \eqref{whatisnwf}. The uniqueness in law
of such a solution, until it hits the boundary, follows since the drift
and the diffusion coefficients are smooth (hence, Lipschitz) inside the
open unit simplex. The law of this process will then be denoted
uniquely by $\operatorname{NWF}(\delta_1, \ldots, \delta_n)$.

Equivalently this process can be identified by its Markov generator.
Expanding $\tsigma\tsigma'$ and using the fact that $\sum_{i=1}^n
x_i=1$, we get
%
\begin{equation}\label{genwf}
\mathcal{A}_n = \frac{1}{2}\sum_{i,j=1}^n x_i ( 1\{i=j\} -
x_j ) \frac{\partial^2}{\partial x_i\, \partial x_j} - \sum
_{i=1}^n \frac{1}{2} (\delta_i - \delta_0 x_i )\frac
{\partial}{\partial x_i},
\end{equation}
which identifies \eqref{cladogen} as the generator for $\operatorname{NWF}(1/2,1/2,1/2)$.

In this text we focus on properties of NWF models as a family of
diffusions on the unit simplex and explore some of their properties
that are important in the context of the Markov chain model on Cladograms.

\textit{Part} (1). We show that, just like Wright--Fisher
diffusions (see~\cite{vsm}), the NWF processes can be recovered from a
far simpler class of models, the Bessel-square (BESQ) processes with
negative dimensions. A comprehensive treatment of BESQ processes can be
found in the book by Revuz and Yor~\cite{RY}. This family of
one-dimensional diffusions is indexed by a single real parameter
$\theta$ (called the dimension) and are solutions of the stochastic
differential equations
%
\begin{equation}\label{besqintro}
Z(t)= x + 2 \int_0^t \sqrt{\abs{Z(s)}}\,d\beta(s) + \theta t, \qquad
x \ge0, t\ge0,
\end{equation}
where $\beta$ is a one-dimensional standard Brownian motion. We denote
the law of this process by $Q^\theta_x$. It can be shown that the
above SDE admits a unique strong solution until it hits the origin. The
classical model only admits paramater $\theta$ to be nonnegative.
However, an extension, introduced by G\"oing-Jaeschke and Yor \cite
{yornbesq}, allows the parameter $\theta$ to be negative. It is
important to note that $Q_x^\theta$ is the diffusion limit of a
Galton--Watson branching process with a $\abs{\theta}$ rate of
immigration (for $\theta\ge0$) or emigration (for $\theta<0$).

In Section~\ref{sec:timechange} we show that the $\operatorname{NWF}(\delta_1,\ldots
, \delta_n)$ law, starting at $(x_1,\ldots,x_n)$, can be recovered
via a stochastic time-change from a collection of $n$ independent
processes with laws $Q_{x_i}^{-2\delta_i}$, $i=1,\ldots,n$, and
dividing each coordinate by the total sum. For the corresponding
discrete models this is usually referred to as Poissonization.

In this article we utilize this relationship to infer several
properties about the NWF processes. For example, we prove that these
diffusions, almost surely, hit the boundary of the simplex. We derive
the explicit exit density supported on the union of the boundary walls
in Theorem~\ref{thm:exitnwf}.

\textit{Part} (2). We also prove an interesting duality
relationship between WF and NWF models. To describe the duality
relationship we let the NWF continue in the lower dimensional simplex
when any of the coordinates hit zero. Thus, every time a coordinate
hits zero, the dimension of the process gets reduced by one, and
ultimately the process is absorbed at the scalar one. Such a process
can be obtained by running a WF model with appropriate parameters that
initially starts with dimension one and value $1$. At independent
random times, the dimension of the process increases by one, and the
newly added coordinate is initialized at zero. Finally we condition on
the values of the process at a chosen random time. The resulting
process, backwards in time and suitably time-changed, is the original
NWF model.

\textit{Part} (3). The time that the NWF process takes to
exit the simplex is a crucial quantity due to a reason which we
describe below. We keep our exposition mostly verbal without going into
too much detail since the details require considerable formalism from
the theory of continuum trees and will be discussed elsewhere. In \cite
{palCT} we show how {Part (1)} points toward a Poissonization
of the entire Aldous Markov chain, which is simpler for considering
scaled limits. The Poissonized version of the Markov chain on $n$-leaf
Cladograms stipulates: every existing leaf has an exponential clock of
rate $2$ attached to it which determines the instances of their deaths,
and every existing edge has an independent exponential clock of rate
$1$ attached to it, at which point the edge is split, and a new pair of
vertices (one of which is a leaf) is introduced. It is an easy
verification that the rates are consistent with the BESQ limit that we
claimed in {Part (1)} above. Hence, one would expect that the
limit of the Poissonized chains on continuum trees, normalized to give
a leaf-mass measure one, and suitably time-changed would give the
conjectured Aldous diffusion. This is the strategy followed in~\cite{palCT}.

Now, the Poissonized chain has some beautiful and interesting
structures. Please see~\cite{A93} for the details about continuum
trees that we use below. A~continuum tree $\ctree$ comes with its
associated (infinite) length measure (analogous to the Lebesgue
measure) and a leaf-mass probability measure, which describes how the
leaves are distributed on it. We will denote the length measure by
$\leb(\ctree)$ and the leaf-mass probability measure by $\mu(\ctree
)$. Suppose we sample $n$ i.i.d. elements from $\mu(\ctree)$ and draw
the tree generated by them, which produces an $n$-leaf Cladogram with
edge-lengths (or, a proper $n$-tree, according to~\cite{A93}). Thus,
by using the fact that the continuum tree is compact, one can
approximate a continuum tree by a sequence of $n$-leaf Cladograms.

Now consider an $n$-leaf Cladogram for a very large $n$, and further
consider $m$ internal branchpoints. For example, in Figure~\ref{fig1}, we have
three branchpoints $\{a,b,c\}$ in a $7$-leaf Cladogram. These
branchpoints generate a \textit{skeleton} subtree of the original tree
and partition the leaves as \textit{internal} or \textit{external} to
the skeleton. The components of the vector of external leaf masses grow
as independent continuous time, binary branching, Galton--Watson
branching processes with a rate of branching/dying $2$ and a rate of
emigration $1$. Note that this is consistent with the diffusion limit
as BESQ with $\theta=-1$. As the Markov chain (Poissonized or not)
proceeds, there comes a time when one of these external leaf masses
gets exhausted. When this happens, one of the internal branch points
becomes a leaf. The distribution of every coordinate of external
leaf-masses at this exit time is derived in Part (2). Until
this time, supported on the skeleton, new subtrees can grow and decay.
We show, in~\cite{palCT}, that the dynamics of the sizes of these
subtrees on the internal part can be modeled as the age process of a
chronological splitting tree. Chronological splitting trees are a
special kind of biological tree, where an individual lives up to a
certain (possibly nonexponential) lifetime and produces children at
rate one during that lifetime. Her children behave in an identical
manner with an independent and identically distributed lifetime of
their own. The age process refers to the point process of current ages
of the existing members in the family. More details about splitting
trees can be found in the article by Lambert~\cite{lambertAOP}.

When one of the internal vertices gets \textit{exposed}, the above
dynamics breaks down, and we need to find a slightly different set of
internal vertices to proceed. Hence, it is important to derive
estimates of the times at which this change happens.
We provide quantitative bounds on the value of this stopping time under
the special situation of symmetric choice of parameters, which is the
case at hand.

The article is divided as follows. Our main tool in this analysis is to
establish a relationship between NWF processes and Bessel-square
processes of negative dimensions, much in the spirit of Pal \cite
{vsm}. This has been done in Section~\ref{sec:timechange} where we
also establish Theorem~\ref{mainthm}. The relevant results about BESQ
processes have been listed in Section~\ref{sec:besq}. Most of these
results are known, and appropriate citations have been provided. Proofs
of the rest can be found in the \hyperref[appm]{Appendix}. Exact computations of exit
density from the simplex have been done in Section \ref
{sec:exitdensity}. Estimates of the exit time have been established in
Section~\ref{sec:exittime}.

\section{Some results about BESQ processes}\label{sec:besq}
The Bessel-square processes of negative dimensions $-\theta$, where
$\theta\ge0$, are one-dimensional diffusions which are the unique
strong solution of the SDE
%
\begin{equation}\label{sdenbesq}
X(t) = x - \theta t + 2\int_0^t \sqrt{X(s)}\,d\beta(s),  \qquad  t \le T_0,
\end{equation}
where $T_0$ is the first hitting time of zero for the process $X$, and
$x$ is a positive constant. The process is absorbed at zero.
We will denote the law of this process $Q_x^{-\theta}$ just as BESQ of
a positive dimension $\theta$ will be denoted by $Q^{\theta}_x$.

The following collection of results is important for us. All the proofs
can be found in the article by G\"oing-Jaeschke and Yor~\cite{yornbesq}.

\begin{lemma}[(Time-reversal)]\label{lem:trev} For any $\theta> -2$ and
any $x >0$, $Q_x^{-\theta}(T_0 < \infty)=1$, while for $\theta\ge
2$, one has $Q^\theta_x(T_0 < \infty)=0$.

Moreover the following equality holds in distribution:
%
\begin{equation}\label{Qtreversal}
\bigl( X(T_0 - u), u \le T_0 \bigr) = \bigl( Y(u), u \le L_x
\bigr),
\end{equation}
where $Y$ has law $Q^{4+\theta}_0$, and $L_x$ is the last hitting time
of $x$ for the process~$Y$.

In particular:
\begin{longlist}[(ii)]
\item[(i)] Both $L_x$ and $T_0$ are distributed as $x/2G$, where $G$
is a Gamma random variable with parameter $(\theta/2 +1)$.
\item[(ii)] The transition probabilities $p_t^\theta(x,y)$ for $x,y
>0$ satisfy the identity
\[
p_t^{-\theta}(x,y)=p_t^{4+\theta}(y,x).\vadjust{\goodbreak}
\]
\end{longlist}
\end{lemma}

The following results have been proved in the \hyperref[appm]{Appendix}.

\begin{lemma}\label{scale}
The scale function for $Q^{-\theta}$, $\theta\ge0$, is given by the function
\[
s(x)= x^{\theta/2+1},  \qquad  x\ge0.
\]
Moreover:
\begin{longlist}[(ii)]
\item[(i)] The origin is an exit boundary for the diffusion and not an entry.
\item[(ii)] The change of measure
\[
x^{-\theta/2-1}Q_x^{-\theta} ( X(t)^{\theta/2+1}1 ( \cdot
) )
\]
on the $\sigma$-algebra generated by the process up to time $t$ is
consistent for various $t$ and is the law of $Q_x^{4+\theta}$. Thus,
we say $Q_x^{4+\theta}$ is $Q_x^{-\theta}$ conditioned never to hit zero.
\end{longlist}
\end{lemma}

The previous fact is the generalization of the well-known observation
that Brownian motion, conditioned never to hit the origin, has the law
of the three-dimensional Bessel process.

\begin{lemma}\label{logct}
Let $\{Z(t), t\ge0\}$ denote a BESQ process of dimension $\theta$
for some $\theta> 2$. Then
\[
\lim_{\epsilon\rightarrow0}\frac{1}{\log(1/\epsilon)} \int
_{\epsilon}^t \frac{du}{Z(u)}=\frac{1}{\theta-2}   \qquad \mbox{for
all } t > 0.
\]
\end{lemma}
%
%
%

\section{Changing and reversing time}\label{sec:timechange}

Our objective in this section is to establish a time-reversal
relationship between NWF and WF models.

\begin{theorem}\label{bestimechange}
Let $z_1, \ldots, z_n$ and $\theta_1,\ldots, \theta_n$ be
nonnegative constants. Let $Z=(Z_1, \ldots,Z_n)$ be a vector of $n$
independent BESQ processes of dimensions $-\theta_1, \ldots, -\theta
_n$, respectively, starting from $(z_1, \ldots, z_n)$. Let $\zeta$ be
the sum $\sum_{i=1}^n Z_i$.

Define
\[
T_i= \inf\{ t\ge0\dvtx  Z_i(t)=0 \},  \qquad \tau=\bigwedge
_{i=1}^n T_i.
\]

Then, there is an $n$-dimensional diffusion $\mu$, satisfying the SDE
in \eqref{whatisnwf} for $\operatorname{NWF}(\theta_1/2, \ldots, \theta_n/2)$, for
which the following equality holds:
%
\begin{equation}\label{bestime}
Z_i(t\wedge\tau) = \zeta(t\wedge\tau) \mu_i ( 4C_t ),
 \qquad 1\le i \le n, \qquad  C_t=\int_0^{t\wedge\tau} \frac{ds}{\zeta(s)}.
\end{equation}
Thus, in particular, equation \eqref{whatisnwf} admits a weak solution
for all nonnegative parameters $(\delta_1, \ldots, \delta_n)$.
%
\end{theorem}

\begin{pf} The proof is almost identical to the case of WF model as
shown in~\cite{vsm}, Proposition 11, with obvious modifications.\vadjust{\goodbreak} For
example, unlike the WF case, the time-change clock is no longer
independent of the NWF process. We outline the basic steps below.

We know from \eqref{sdenbesq} that
\[
d Z_i(t \wedge\tau) = - \theta_i d(t\wedge\tau) + 2 \sqrt{Z_i}\,d\beta_i( t \wedge\tau),  \qquad  i=1,2,\ldots,n.
\]
Define $\theta_0=\sum_{i=1}^n \theta_i$. Let $V_i(t)=Z_i/\zeta(t)$
for $t \le\tau$. Then by It\^o's rule, we get
%
\begin{equation}\label{msde}
dV_i(t\wedge\tau)=-\zeta^{-1} [\theta_i - \theta_0 V_i
]\,d(t\wedge\tau) + \sqrt{{V_i}(1-V_i)}\,dM_i(t),
\end{equation}
where
%
\begin{equation}\label{howtodefmi}
dM_i(t)= \frac{2\zeta^{-1/2}}{\sqrt{1- V_i}}\sum_{j=1}^n \bigl(1\{
i=j\} - \sqrt{V_iV_j} \bigr)\,d\beta_j(t\wedge\tau),
\end{equation}
and $\langle M_i \rangle(t)=4C_t$.

Let $\{\rho_u, u\ge0 \}$ be the inverse of the increasing function
$4C_t$. Applying this time-change to the SDE for $V_i$ in \eqref
{msde}, we get
%
\begin{equation}\label{nuwt}
d\mu_i(t) =- \tfrac{1}{4} [\theta_i - \theta_0 \mu_i ]\,dt
+ \sqrt{\mu_i(1-\mu_i)}\tw_i(t),
\end{equation}
where $\tw_i$ is the Dambis--Dubins--Schwarz (DDS; see~\cite{RY}, page~181) Brownian motion associated with $M_i$. This turns out
to be the SDE for $\operatorname{NWF}(\theta_1/2,\break \ldots,  \theta_n/2)$.
\end{pf}

%

Let $\theta_1, \theta_2, \ldots, \theta_n$ be nonnegative and
$z_1, z_2, \ldots, z_n$ be positive constants. For $i=1,2,\ldots,n$
define independent random variables $(G_1, \ldots, G_n)$ where $G_i$
is distributed as $\operatorname{Gamma}(\theta_i/2 + 1)$. Let
%
\begin{equation}\label{whatisri}
R_i=\frac{z_i}{2 G_i}, \qquad i=1,2,\ldots,n.
\end{equation}
Also, independent of $(G_1, \ldots, G_n)$, let $Y_1, Y_2, \ldots,
Y_n$ be $n$ independent BESQ processes of positive dimensions
$(4+\theta_1),(4+\theta_2), \ldots, (4 + \theta_n)$, respectively,
all of which are starting from zero.

For any permutation $\pi$ of $n$ labels, condition on the event
%
\begin{equation}\label{condperm}
R_{\pi_1} > R_{\pi_2} >\cdots> R_{\pi_n}     \quad  \mbox{and
let} \quad  R^*=R_{\pi_2}.
\end{equation}

We now construct the following $n$ dimensional process $(X_1, \ldots, X_n)$:
%
\begin{equation}\label{whatisxi}
X_i(t)=Y_i \bigl( (t- R^* + R_i)^+ \bigr), \qquad t\ge0.
\end{equation}
Notice that at time $t=0$, every $X_i$ is at zero except the $\pi_1$th.

Let $S(t)$ denote the total sum process $\sum_{i=1}^n X_i(t)$. Note
that $S(t) >0$ for all $t \ge0$ with probability one. Define the process
%
\begin{equation}\label{whatisct}
C_t:=\int_{0}^{t}\frac{du}{S(u)}, \qquad  t > 0.
\end{equation}
The process $C_t$ is finite almost surely for every $t$ (unfortunately,
we cannot define $R^*=R_{\pi_1}$ precisely because $C_t$ will be
infinity; see Lemma~\ref{logct}). Let $A$ denote the inverse function
of the continuous increasing function $4C$. That is,
%
\begin{equation}\label{whatisa}
A_t = \inf\{ u \ge0\dvtx  4C_u \ge t \},  \qquad  t\ge0.
\end{equation}

\begin{lemma}\label{dimWF} There is an $n$-dimensional diffusion $\nu
$ such that the following time-change relationship holds:
%
\begin{equation}\label{getxi}
\nu_i(t)=\frac{X_i}{S}(A_t)  \quad\mbox{or}\quad X_i(t)= S(t) \nu
_i ( 4 C_t ), \qquad  t \ge0.
\end{equation}
The distribution of $\nu$ is supported on the unit simplex
\[
\usimp_n= \{ x_i \ge0\dvtx  x_1 + x_2 +\cdots+ x_n=1 \}.
\]
Conditional on the values of $G_1, \ldots, G_n$ and the process $S$,
the law of $\nu$ can be described as below.

Let $\pi$ be any permutation of $n$ labels. On the event $R_{\pi_1} >
R^*=R_{\pi_2} >\cdots> R_{\pi_n}$. Let $V_2 <\cdots< V_{n}$ be
defined by
\[
A_{V_i}= R^* - R_{\pi_{i}}   \quad    \mbox{or, equivalently} \quad
4C_{R^*-R_{\pi_i}}=V_i.
\]
Note that $V_2=0$.

For $i\ge2$ and $V_{i} \le t \le V_{i+1}$, the process $\nu$ is zero
on all coordinates except $(\pi_1, \ldots, \pi_i)$. The process $\nu
(\pi_1, \ldots, \pi_i)$, given the history of the process till time
$V_i$ (and the $G_i$'s and $S$), is distributed as the classical
Wright--Fisher diffusion starting from
\[
\frac{1}{S}(X_{\pi_1}, \ldots, X_{\pi_i}) (A_{V_i}
)=\frac{1}{S}(X_{\pi_1}, \ldots, X_{\pi_i}) (R^*-R_{\pi_i}
),
\]
and with parameters $(\gamma_{\pi_1}, \ldots, \gamma_{\pi_i} )$ where
\[
\gamma_j = \theta_j/2 + 2,  \qquad  j=1,2,\ldots,n.
\]
\end{lemma}

\begin{pf} The Gamma random variables $G_1, \ldots, G_n$ are
independent of the BESQ process $Y_1, \ldots, Y_n$.
Thus, conditional on $G_1, \ldots, G_n$, the vector of processes
$(X_1, \ldots, X_n)$ has the following description. For
\[
R^*- R_{\pi_i}\le t \le R^* - R_{\pi_{i+1}}, \qquad  i\ge2,
\]
all coordinates other than the $\pi_1$th, $\pi_2$th, \ldots, $\pi
_i$th are zero. And, $(X_{\pi_1}, \ldots,\break X_{\pi_i})$, conditioned
on the past, are independent BESQ processes of dimensions $(4+\theta
_{\pi_1}, \ldots, 4+\theta_{\pi_i})$ and starting from $(X_{\pi
_1}, \ldots, X_{\pi_i})(R^*- R_{\pi_i})$.

Thus, on this interval of time, the existence of the process $\nu$,
identifying its law as the WF law, and the claimed independence from
the process $S$, all follow from~\cite{vsm}, Proposition~11. The proof
of the lemma now follows by combining the argument over the distinct intervals.
\end{pf}

\begin{lemma}\label{timerevdimWF}
Consider the set-up in \eqref{whatisri}, \eqref{whatisxi} and \eqref
{whatisa}. Let $Z_1, Z_2, \ldots, Z_n$ be $n$ stochastic processes
defined such that $\{Z_i(t), 0\le t \le R^*\}$ is the time-reversal
of the process $\{X_i(t), 0\le t \le R^*\}$, conditioned on
$X_i(R^*)= z_i$. That is, conditioned on $X_i(R^*)=z_i$ for every $i$,
\[
Z_i(t) = X_i(R^*-t)=Y_i(R_i-t)^+   \qquad \mbox{for } 0 \le t \le R^*.
\]
Then $(Z_1, \ldots, Z_n)$ are independent BESQ processes of dimensions
$-\theta_1, \ldots, \break -\theta_n$, starting from $z_1, \ldots, z_n$,
and absorbed at the origin.
\end{lemma}

\begin{pf}
It suffices to prove the following:

\begin{claim*} Let $\{Y(t), t\ge0\}$ denote a BESQ
process of dimension $(4+\theta)$ starting from $0$. Fix a $z >0$. Let
$T$ be distributed as $z/2G$, where $G$ is a Gamma random variable with
parameter $(\theta/2+1)$. Then, conditioned on $T=l$ and $Y(l)=z$, the
time-reversed process $\{Y((l-s)^+), 0\le s < \infty\}$ is distributed
as $Q_z^{-\theta}$, absorbed at the origin, conditioned on $T_0=l$.
Here $T_0$ is the hitting time of the origin for $Q_z^{-\theta}$.
\end{claim*}

Once we prove this claim, the lemma follows since the law of $T_0$ is
exactly $z/2G$. See Lemma~\ref{lem:trev}.

\begin{pf*}{Proof of Claim}
For the case of $\theta=0$, this
is proved in~\cite{besselbridge}, page~447. The general proof is
exactly similar and we outline just the steps and give references
within~\cite{besselbridge} for the details.

For any $\theta\in\rr$, $t >0$, $x,y \ge0$, let $Q^{\theta
,t}_{x\rightarrow y}$ denote the law of the BESQ bridge of dimension
$\theta$, length $t$, from points $x$ to $y$. That is to say, if $Y$
follows $Q_x^\theta$, then $Q^{\theta,t}_{x\rightarrow y}$ is the law
of the process $\{ Y(s), 0\le s\le t \}$ conditioned on the event $\{
Y(t)=y\}$.

Now, BESQ bridges satisfy time-reversal~\cite{besselbridge}, page~446.
Thus, if we define $\widehat{P}$ to be the $P$-distribution of a
process $\{ X(t-s), 0\le s\le t\}$, then \mbox{$Q^{\theta,t}_{x\rightarrow
y}=\widehat Q^{\theta,t}_{y\rightarrow x}$}.

We consider the case when the dimension is $(4+\theta), \theta\ge
0$, $x=0, y=z >0$. Then
\[
Q^{4+\theta,t}_{z\rightarrow 0}=\widehat Q^{4+\theta,t}_{0\rightarrow z}.
\]

Now, from Lemma~\ref{scale} (also see~\cite{besselbridge}, Section 3, page~440), we know that $Q^{4+\theta}_z$ is $Q^{-\theta
}_z$ conditioned never to hit zero (or equivalently, $Q^{-\theta}_z$
can be interpreted as $Q_z^{4+\theta}$ conditioned to hit zero). Since
the origin is an exit distribution for $Q^{-\theta}_z$ and not an
entry (Lemma~\ref{scale}; see~\cite{besselbridge}, page~441, for the
details of these definitions), the conditional law $Q^{4+\theta
,t}_{z\rightarrow0}$ is nothing but $Q^{-\theta}_z$, conditioned on
$T_0=t$. This completes the proof.
\end{pf*}
\noqed
\end{pf}

The following is a more precise statement.

Let $(z_1, \ldots, z_n)$ be a point in the $n$-dimensional unit
simplex $\usimp_n$. Fix $n$ nonnegative parameters\vadjust{\goodbreak} $\delta_1, \ldots
, \delta_n$. Let $G_1, \ldots, G_n$ denote $n$ independent Gamma
random variables with parameters $\delta_1+1, \ldots, \delta_n+1$,
respectively. Define $R_i=z_i/2G_i$.

For any permutation $\pi$ of $n$ labels, condition on the event
$R_{\pi_1} > R_{\pi_2} >\cdots> R_{\pi_n}$, and let $R^*=R_{\pi_2}$.

Define the continuous process $S$ by prescribing $S(0)=Z_1(R_{\pi
_1}-R^*)$ where $Z_1$ is distributed as $Q_0^{4+2\delta_{\pi_1}}$,
and for any $t$ such that
\[
R^*-R_{\pi_i} \le t \le R^*-R_{\pi_{i+1}},\qquad i\ge2, \qquad  R_{\pi
_{n+1}}=0.
\]
Given the history, the process is distributed as a Bessel-square
process of dimension $\sum_{j=1}^i (4+ 2\delta_{\pi_j})$ starting
from $S(R^*-R_{\pi_i})$.

Define the stochastic clocks
\[
C_t = \int_{0}^t \frac{du}{S(u)}, \qquad\widehat C_t = \int
_{R^*-t}^{R^*} \frac{du}{S(u)}, \qquad  0\le t\le R^*,
\]
and let $\widehat A_t$ denote the inverse function of $4\widehat C_t$.
Let $V_2 <\cdots< V_{n}$ be defined by $4C_{R^*-R_{\pi_i}}=V_i$.
Note that $V_2=0$. The $4$ is a standardization constant that appears
due to the factor of $2$ in the diffusion coefficient in \eqref{besqintro}.

Define an $n$-dimensional process $\nu$, given $R_1, \ldots, R_n$,
and the process $S$.
For $i\ge2$ and $V_{i} \le t \le V_{i+1}$, the process $\nu$ is zero
on all coordinates, except possibly at indices $(\pi_1, \ldots, \pi
_i)$. At time zero, the process starts at the vector that is $1$ in the
$\pi_1$th coordinate and zero elsewhere.

Conditioned on the history till time $V_i$, the process $\{\nu(\pi_1,
\ldots, \pi_i)(t), V_i \le t \le V_{i+1}\}$ is distributed as the
classical Wright--Fisher diffusion, starting from $\nu(\pi_1, \ldots
, \pi_i)(V_i)$ and with parameters $(\gamma_{\pi_1}, \ldots, \gamma
_{\pi_i} )$, where
\[
\gamma_j = \delta_j + 2,  \qquad  j=1,2,\ldots,n.
\]

Finally, consider the conditional law of the process, conditioned on
the event
\[
S(R^*) \nu_i(4C_{R^*})= z_i   \qquad \mbox{for all } i=1,2,\ldots,n.
\]

\begin{theorem}\label{mainthm}
Define the time-reversed process
\[
\mu(t) = \nu( \widehat A \circ4C_{R^*-t} ),
\]
where $\circ$ denotes composition. Then this conditional stochastic
time-reversed process, until the first time any of the coordinates hit
zero, has a marginal distribution (when $G_i$'s and $S$ are integrated
out) $\operatorname{NWF}(\delta_1, \ldots, \delta_n)$ starting from $(z_1, \ldots, z_n)$.
\end{theorem}

%
%
%
%
%
%

\begin{pf}
We start with given values of $R_{\pi_1} > R_{\pi_2} >\cdots>
R_{\pi_n}$ and the process $S$ and apply equation \eqref{getxi} in
Lemma~\ref{dimWF} to obtain the processes $(X_1, \ldots, X_n)$,
defined by
\[
X_i(t) = S(t) \nu_i(4C_t),  \qquad  0\le t \le R^*.
\]
Then, the vector $(X_1, X_2, \ldots, X_n)$ has the law prescribed by
\eqref{whatisxi}.\vadjust{\goodbreak}

Now we apply Lemma~\ref{timerevdimWF} to obtain $(Z_1, \ldots, Z_n)$
by conditioning $(X_1, \ldots, X_n)$ and reversing time.
Finally the construction in Theorem~\ref{bestimechange} gives us the
vector $(\mu_1, \ldots, \mu_n)$ from $(Z_1, \ldots, Z_n)$, as desired.
\end{pf}

\section{Exit density}\label{sec:exitdensity}

Let $Z_1, Z_2, \ldots, Z_n$ be independent BESQ processes of
dimensions $-\theta_1, \ldots, -\theta_n$, where each $\theta_i \ge
0$. We assume that at time zero, the vector $\vZ=(Z_1, \ldots, Z_n)$
starts from a point $\mathbf{z}=(z_1, \ldots, z_n)$ where every $z_i
> 0$. Define $T_i$ to be the first hitting time of zero for the process
$Z_i$, and let $\tau=\bigwedge_{i} T_i$ denote the first time any
coordinate hits zero. We would like to determine the joint distribution
of $(\tau, \vZ(\tau))$.

Note that since each $T_i$ is a continuous random variable, the minimum
is attained at a unique $i$. Thus, for a fixed $1\le i \le n$,
conditioned on the event $\tau=T_i$, the distribution of $Z_i(\tau)$
is the unit mass at zero, and the distribution of every other $Z_j(\tau
)$ is supported on $(0, \infty)$. Now, let $h_i$ denote the density of
the stopping time $T_i$ on $(0,\infty)$, and let $q_t^{-\theta}$
refer to the transition density of $Q^{-\theta}$. It follows that for
any $a_j > 0$, $j\neq i$, we get
%
\begin{eqnarray}
&&P \bigl( \tau=T_i, \tau\le t, Z_j(\tau) \ge a_j  \mbox{ for all
$j \neq i$}\bigr)\nonumber\\
&& \qquad =P \bigl( T_i \le t, T_j > T_i, Z_j(T_i) \ge a_j  \mbox{ for all
$j\neq i$} \bigr)\nonumber\\
&& \qquad = \int_0^t h_i(s) \prod_{j\neq i} P \bigl( T_j > s, Z_j(s) \ge a_j
\bigr)\,ds= \int_0^t h_i(s) \prod_{j\neq i}P \bigl( Z_j(s) \ge a_j
\bigr)\,ds  \nonumber\\
 \eqntext{\mbox{since $a_j > 0$}}\\
&& \qquad = \int_0^t h_i(s)\biggl [ \prod_{j\neq i} \int_{a_j}^\infty
q_s^{-\theta_j}(z_j, y_j)\,dy_j \biggr]\,ds.\nonumber
\end{eqnarray}

Our first job is to find closed form expressions of the integral above.
To do this we start by noting that $T_i$ is distributed as $z_i/ 2G_i$
(see Lemma~\ref{lem:trev}), where $G_i$ is a Gamma random variable
with parameter $(4+\theta_i)/2-1=\theta_i/2 + 1$. That is, the
density of $G_i$ is supported on $(0,\infty)$ and is given by
\[
\frac{y^{\theta_i/2}}{\Gamma(\theta_i/2 + 1)}e^{-y}.
\]
It follows that
\[
h_i(s)= \frac{(z_i/2)^{\theta_i/2+1}}{\Gamma(\theta_i/2 + 1)}
s^{-\theta_i/2 -2} e^{-z_i/2s},  \qquad  0 \le s < \infty.
\]

On the other hand, it follows from time reversal (Lemma \ref
{lem:trev}) that $q_s^{-\theta_j}(z_j, \break y_j)= q_{s}^{4+\theta_j}(y_j,
z_j)$. For any positive $a$, the transition density $q^{a}_s(y,z)$ is
explicitly known (see, e.g.,~\cite{vsm}) to be $s^{-1}f(z/s, a, y/s)$,
where $f(\cdot,k,\lambda)$ is the density of a noncentral Chi-square
distribution with $k$-degrees of freedom and a noncentrality parameter
value $\lambda$. In particular, it can be written as a Poisson mixture
of central Chi-square (or, Gamma) densities. Thus we have the following
expansion:
%
\begin{equation}\label{tranden}
 \qquad q_s^{-\theta_j}(z_j, y_j)= q_{s}^{4+\theta_j}(y_j, z_j)=s^{-1}\sum
_{k=0}^{\infty} e^{-y_j/2s}\frac{(y_j/2s)^k}{k!} g_{\theta_j+4+2k}(z_j/s),
\end{equation}
where $g_r$ is the Gamma density with parameters $(r/2, 1/2)$. That is,
\[
g_r(x)= \frac{2^{-r/2} x^{r/2-1}}{\Gamma(r/2)} e^{-x/2},  \qquad  x\ge0.
\]

Now, define
\[
\ymini=\sum_{j\neq i} y_j,  \qquad \tmini= \sum_{j\neq i}\theta
_j, \qquad \zmini=\sum_{j\neq i} z_j.
\]
Thus
\begin{eqnarray*}
 &&h_i(s)\prod_{j\neq i} q_s^{-\theta_j}(z_j, y_j)\\
 && \qquad = \frac
{(z_i/2)^{\theta_i/2+1}}{\Gamma(\theta_i/2 + 1)} s^{-\theta_i/2 -2}
e^{-z_i/2s} \prod_{j\neq i} s^{-1}\sum_{k=0}^{\infty}
e^{-y_j/2s}\frac{(y_j/2s)^k}{k!} g_{\theta_j+4+2k}(z_j/s)\\
 && \qquad =\frac{(z_i/2)^{\theta_i/2+1}}{\Gamma(\theta_i/2 + 1)} s^{-\theta
_i/2 -2} e^{-z_i/2s} \\
 && \qquad  \quad  {}\times\prod_{j\neq i} s^{-1}\sum_{k=0}^{\infty}
e^{-y_j/2s}\frac{(y_j/2s)^k}{k!} \frac{2^{-\theta_j/2-2-k}
(z_j/s)^{\theta_j/2+k+1}}{\Gamma(\theta_j/2 + 2 +k)} e^{-z_j/2s}\\
 && \qquad =\frac{(z_i/2)^{\theta_i/2+1}}{\Gamma(\theta_i/2 + 1)} s^{-\theta
_i/2 - 2-(n-1)} e^{-z_i/2s}\\
 && \qquad  \quad  {}\times e^{-(\ymini+ \zmini)/2s}2^{-\tmini/2-2(n-1)} \prod_{j\neq i} \sum
_{k=0}^{\infty} \frac{(y_j/2s)^k}{k!} \frac{2^{-k} (z_j/s)^{\theta
_j/2+k+1}}{\Gamma(\theta_j/2 + 2 +k)}.
\end{eqnarray*}

We now exchange the product and the sum in the above. We will need some
more notations for a compact representation.
For any two vectors $a$ and $b$, denote by
\[
a^b = \prod_{i} a_i^{b_i},  \qquad  a!=\prod_{i} a_i!.
\]
Also let $\vtheta_i, \vy_i, \mathbf{z}_i$ stand for the vectors
$(\theta_j, j\neq i)$, $(y_j, j\neq i)$ and $(z_j, j\neq i)$,
respectively.

Let $\vk$ denote the vector $(k_j, j\neq i)$, where every $k_j$
takes any nonnegative integer values. Let $\vk'1$ be the sum of the
coordinates of $\vk$. Then
\begin{eqnarray*}
&&\prod_{j\neq i}  \sum_{k=0}^{\infty} \frac{(y_j/2s)^k}{k!} \frac
{2^{-k} (z_j/s)^{\theta_j/2+k+1}}{\Gamma(\theta_j/2 + 2 +k)} \\
&& \qquad = \sum_{N=0}^{\infty} (4s)^{-N} s^{-\tmini/2 - N - (n-1)} \mathbf
{z}_i^{\vtheta_i/2+1}\sum_{\vk'1=N} \frac{\vy_i^\vk}{\vk!} \frac
{\mathbf{z}_i^{\vk}}{\prod_{j\neq i} \Gamma(\theta_j/2 + 2 + k_j)}.
\end{eqnarray*}

Thus, combining the expressions, we get
%
\begin{eqnarray}\label{inter1}
&&h_i(s) \prod_{j\neq i} q_s^{-\theta_j}(z_j, y_j)\nonumber\\
&& \qquad =\frac{z_i^{\theta
_i/2+1}}{\Gamma(\theta_i/2 + 1)}2^{-\theta_i/2 -1 -\tmini
/2-2(n-1)} \\
&& \qquad  \quad {}\times s^{-\theta_i/2 - 2-(n-1)} e^{-z_i/2s} e^{-(\ymini+ \zmini)/2s}
\sum_{N=0}^{\infty} 4^{-N} s^{-\tmini/2 - 2N - (n-1)}B_N,\nonumber
\end{eqnarray}
where
\[
B_N= \mathbf{z}_i^{\vtheta_i/2+1}\sum_{\vk'1=N} \frac{\vy_i^\vk
}{\vk!} \frac{\mathbf{z}_i^{\vk}}{\prod_{j\neq i} \Gamma(\theta
_j/2 + 2 + k_j)}.
\]

We can now integrate over $s$ in \eqref{inter1} to obtain
\begin{eqnarray*}
\int_0^\infty h_i(s) \prod_{j\neq i} q_s^{-\theta_j}(z_j, y_j)\,ds=
\sum_{N=0}^\infty B'_N \int_0^\infty s^{-a_N} e^{-b/s}\,ds,
\end{eqnarray*}
where
%
\begin{eqnarray}
B_N' &=& \frac{z_i^{\theta_i/2+1}}{\Gamma(\theta_i/2 +
1)}2^{-\theta_0/2-2n+1} 4^{-N} B_N,  \qquad \theta_0=\sum_{i=1}^n
\theta_i,\\
a_N &=& \theta_i/2 + \tmini/2 +2n + 2N= \theta_0/2 + 2n + 2N,\\
b &=& z_i/2 + (\ymini+\zmini)/2 = (\ymini+ z_0)/2,  \qquad  z_0=\sum
_{i=1}^n z_i.
\end{eqnarray}

Now a simple change of variable $w=1/s$ shows
\begin{eqnarray*}
&\displaystyle \int_0^\infty s^{-a_N} e^{-b/s}\,ds = \int_0^\infty
w^{a_N}e^{-bw}w^{-2}\,dw=\int_0^\infty w^{a_N -2} e^{-bw}\,dw,&\\
&\displaystyle \frac{\Gamma(a_N-1)}{b^{a_N-1}} \int_0^{\infty} \frac
{b^{a_N-1}}{\Gamma(a_N-1)} w^{a_N -2} e^{-bw}\,dw = \frac{\Gamma
(a_N-1)}{b^{a_N-1}}.&
\end{eqnarray*}

Since the $i$th coordinate of the exit point is zero, one can define
$y_i=0$ and $y_0=\sum_{j=1}^n y_j= \ymini$ to simplify notation.
Thus we obtain
\begin{eqnarray*}
&&\int_0^\infty h_i(s)\prod_{j\neq i} q_s^{-\theta_j}(z_j, y_j)\,ds\\
 && \qquad =
\sum_{N=0}^\infty\frac{z_i^{\theta_i/2+1}}{\Gamma(\theta_i/2 +
1)}2^{-\theta_0/2-2n+1} 4^{-N}B_N \frac{\Gamma(a_N-1)}{b^{a_N-1}}\\
&&  \qquad  =\frac{z_i^{\theta_i/2+1}\mathbf{z}_i^{\vtheta_i/2+1}}{\Gamma
(\theta_i/2 + 1)}2^{-\theta_0/2-2n+1} \sum_{N=0}^\infty\bigl(
(\ymini+z_0)/2 \bigr)^{-\theta_0/2 - 2n - 2N + 1}\\
&& \qquad  \quad {}\times\Gamma(\theta_0/2 + 2n + 2N-1) 4^{-N}\sum_{\vk'1=N} \frac
{\vy_i^\vk}{\vk!} \frac{\mathbf{z}_i^{\vk}}{\prod_{j\neq i}
\Gamma(\theta_j/2 + 2 + k_j)}\\
&& \qquad =\frac{\mathbf{z}^{\vtheta/2+1}}{\Gamma(\theta_i/2 +
1)}2^{-\theta_0/2-2n+1} \sum_{N=0}^\infty( y_0 + z_0
)^{-\theta_0/2 - 2n - 2N + 1} 2^{\theta_0/2 + 2n + 2N -1}\\
&& \qquad  \quad {}\times\Gamma(\theta_0/2 + 2n + 2N-1) 4^{-N}\sum_{\vk'1=N} \frac
{\vy_i^\vk}{\vk!} \frac{\mathbf{z}_i^{\vk}}{\prod_{j\neq i}
\Gamma(\theta_j/2 + 2 + k_j)}\\
&& \qquad =\frac{\mathbf{z}^{\vtheta/2+1}}{\Gamma(\theta_i/2 + 1)} \sum
_{N=0}^\infty( y_0 + z_0 )^{-\theta_0/2 - 2n - 2N +
1}\\
&& \qquad  \quad {}\times \Gamma(\theta_0/2 + 2n + 2N-1)\sum_{\vk'1=N} \frac{\vy_i^\vk
}{\vk!} \frac{\mathbf{z}_i^{\vk}}{\prod_{j\neq i} \Gamma(\theta
_j/2 + 2 + k_j)}.
\end{eqnarray*}

We have the following result.

\begin{theorem}\label{thm:besexit}
Let $Z_1, Z_2, \ldots, Z_n$ be independent BESQ processes of
dimensions $-\theta_1, \ldots, -\theta_n$, where each $\theta_i\ge
0$. Assume that $Z_i(0)= z_i(0) > 0$, for every~$i$.

The distribution of $(\tau, Z({\tau}))$ is supported on the set
$(0,\infty)\times\bigcup_{i=1}^n H_i$, where $H_i$ is the subspace
orthogonal to the $i$th canonical basis vector $e_i$. That is,
\[
H_i = \{ (y_1, y_2, \ldots, y_n)\dvtx  y_i=0 \}.
\]

\begin{longlist}[(ii)]
\item[(i)] Let $G_i, i=1,2,\ldots,n$ be independent Gamma random
variables with parameters $\theta_i/2 + 1, i=1,2,\ldots, n$.
The law of $\tau$ is the same as that of $\min_{i} \frac{z_i}{2G_i}$ and
\[
P ( \tau=T_i )= P \biggl( \frac{G_i}{z_i} > \frac
{G_j}{z_j}  \mbox{ for all $j\neq i$} \biggr),
\]
where $T_i$ is the first hitting time of $H_i$.\vadjust{\goodbreak}
\item[(ii)] The restriction of the law of the random vector $Z(\tau
)$, restricted to the hyperplane $H_i$, admits a density with respect
to all the variables $y_j$'s, $j\neq i$, which is given by
%
\begin{eqnarray}\label{eq:eden}
&=&\frac{S^{1-\theta_0/2-2n}}{\Gamma(\theta_i/2 + 1)}\prod_{j=1}^n
z_j^{\theta_j/2 + 1}\sum_{N=0}^\infty\Gamma(\theta_0/2 + 2n +
2N-1) S^{- 2N}\nonumber
\\[-8pt]
\\[-8pt]
&&{}\times\sum_{\sum_{j\neq i} k_j=N} \prod_{j\neq i} \frac{(y_j
z_j)^{k_j}}{k_j! \Gamma(\theta_j/2 + 2 + k_j)}.
\nonumber
\end{eqnarray}
Here
\[
S=\sum_{i=1}^n (y_i + z_i), \qquad y_i=0, \qquad\theta_0=\sum_{i=1}^n
\theta_i.
\]
\end{longlist}
\end{theorem}

Using Theorem~\ref{bestimechange}, we get that the exit distribution
of $\operatorname{NWF}(\delta_1, \ldots, \delta_n)$, starting from a point $(z_1,
\ldots, z_n)\in\usimp_n$, is the image under the map
\[
x_i \mapsto\frac{x_i}{\sum_{j=1}^n x_j}, \qquad1\le i\le n,
\]
of the exit density of independent BESQ processes of dimensions
$-\theta_1, \ldots, -\theta_n$, where each $\theta_i=2\delta_i$.

\begin{theorem}\label{thm:exitnwf}
The exit density of $\mu\sim$ $\operatorname{NWF}(\delta_1, \ldots, \delta_n)$
starting from $(z_1, \ldots,  z_n)\in\usimp_n$ is supported on the
set $\bigcup_{i=1}^n F_i$, where $F_i$ is the face $\{ x \in\usimp_n:
x_i=0 \}$, and admits the following description:
\begin{longlist}[(ii)]
\item[(i)] Let $G_i, i=1,2,\ldots,n$, be independent Gamma random
variables with parameters $\delta_i + 1, i=1,2,\ldots, n$.
Then
%
\begin{equation}\label{eq:exithyp}
P ( \mbox{$\mu$ exits through $F_i$} )= P \biggl( \frac
{G_i}{z_i} > \frac{G_j}{z_j}  \mbox{ for all $j\neq i$} \biggr).
\end{equation}

\item[(ii)] Let $\delta$ represent the vector $(\delta_1, \ldots,
\delta_n)$, and let $\delta_0=\sum_{i=1}^n \delta_i$. The exit
distribution of the process $\mu$, restricted to $F_i$, admits a
density with respect to all the variables $x_j$'s, $j\neq i$, which is
given by
%
\begin{equation}\label{eq:exitnwf}
(\delta_i + 1) \sum_{N=0}^\infty\frac{\Gamma(N+n+\delta
_0)}{\Gamma(N+2n + \delta_0)} \sum_{\sum_{j\neq i} k_j=N} \diri
_n(z;\vk+ \delta+\vtwo) \diri_{n-1}(x;\vk+\vone).\hspace*{-35pt}
\end{equation}
Here the inner sum above is over all nonnegative integers $(k_j, j
\neq i)$, such that $\sum_{j\neq i}k_j =N$. The vector $\vk$
represents a vector whose $j$th coordinate is $k_j$ for all $j \neq i$,
and $\vk_i=0$. The vectors $\vk+\delta+\vtwo$ and $\vk+\vone$
represent vector additions of $\vk$, $\delta$ and the vector of all
twos, and $\vk$ and the vector of all ones, respectively. The factor
$\diri_{n-1}$ is a density with respect to the $(n-1)$-dimensional
vector $(x_j, j\neq i)$ with corresponding parameters $(\vk
_{j}+1, j \neq i)$. It can also be interpreted as the conditional
density of the $n$-dimensional $\diri_n(x;\vk+1)$, conditioned on $x_i=0$.
\end{longlist}
\end{theorem}

Note that the density in \eqref{eq:exitnwf} is a mixture of Dirichlet
densities, strikingly similar to those appearing as transition
probabilities of the Wright--Fisher diffusions themselves. See
Griffiths~\cite{griffiths79b}, Barbour, Ethier and Griffiths \cite
{BEG} and Pal~\cite{vsm}.

\begin{pf*}{Proof of Theorem~\ref{thm:exitnwf}} This is a
straightforward integration. We have assumed that $\sum_i z_i=1$.
Thus, $S=1+\sum_j y_j$; define $y_0=\sum_j y_j$, and
\[
x_j = y_j/y_0, \qquad1\le j \le n.
\]
Hence \eqref{eq:eden} simplifies to
%
\begin{eqnarray}
&=&\frac{(1+y_0)^{1-\theta_0/2-2n}}{\Gamma(\theta_i/2 + 1)}\prod
_{j=1}^n z_j^{\theta_j/2 + 1}\sum_{N=0}^\infty\Gamma(\theta_0/2 +
2n + 2N-1) (1+y_0)^{-2 N} \nonumber
\\[-8pt]
\\[-8pt]
&&{}\times y_0^N\sum_{\sum_{j\neq i} k_j=N} \prod_{j\neq i} \frac{(x_j
z_j)^{k_j}}{k_j! \Gamma(\theta_j/2 + 2 + k_j)}.
\nonumber
\end{eqnarray}

Now, to get to formula \eqref{eq:exitnwf} we need to make a
multivariate change of variables. Without loss of generality, let
$i=n$. Then, for any $y\in F_i$, we have $y_{n}=0$. Define the change
of variables
\[
(y_1, \ldots, y_{n-2}, y_{n-1}) \mapsto( y_0,x_1, \ldots,
x_{n-2} ).
\]
In other words, $y_i= y_0 x_i$ for all $i=1,2,\ldots, n-2$ and
$y_{n-1}=y_0(1-x_1-\cdots- x_{n-2})$. The determinant of the
well-known Jacobian matrix is given by $y_0^{n-2}$.

Thus, the density of $(x_1, \ldots, x_n)$ restricted to $F_i$ is given by
%
\begin{eqnarray}\label{denform1}
&&\frac{1}{\Gamma(\theta_i/2 + 1)} \prod_{j=1}^n z_j^{\theta_j/2 +
1}\sum_{N=0}^\infty\Gamma(\theta_0/2 + 2n + 2N-1)\nonumber
\\
&& \qquad {}\times\int_0^\infty y^{N+n-2}(1+y)^{1-\theta_0/2-2n-2N}\,dy \\
&& \qquad {}\times\sum
_{\sum_{j\neq i} k_j=N} \prod_{j\neq i} \frac{(x_j z_j)^{k_j}}{k_j!
\Gamma(\theta_j/2 + 2 + k_j)}.
\nonumber
\end{eqnarray}
The following formula is easily verifiable for $\alpha\ge0$, $\beta>
\alpha+1$:
\[
\int_0^\infty y^{\alpha} (1+y)^{-\beta}\,dy=\int_0^1 x^{\beta
-\alpha-2}(1-x)^{\alpha}\,dx=B(\alpha+1,\beta-\alpha-1),
\]
where $B$ refers to the Beta function.\vadjust{\goodbreak}

In other words, \eqref{denform1} reduces to
%
\begin{eqnarray}\label{denform2}
&&\frac{1}{\Gamma(\theta_i/2 + 1)}\prod_{j=1}^n z_j^{\theta_j/2 +
1}\sum_{N=0}^\infty\Gamma(\theta_0/2 + 2n + 2N-1)\nonumber
\\[-9pt]
\\[-9pt]
&& \qquad {}\times B(N+n-1, N+n+\theta_0/2) \sum_{\sum_{j\neq i} k_j=N} \prod_{j\neq
i} \frac{(x_j z_j)^{k_j}}{k_j! \Gamma(\theta_j/2 + 2 + k_j)}.
\nonumber
\end{eqnarray}

We now change $\theta_i/2$ to $\delta_i$ and rewrite the above
expression in terms of Dirichlet densities. We use the notations in the
statement of Theorem~\ref{thm:exitnwf}: the vector $\vk$ represents a
vector whose $j$th coordinate is $k_j$ for all $j \neq i$, and $\vk
_i=0$. The vectors $\vk+\delta+\vtwo$ and $\vk+\vone$ represent
vector additions of $\vk$, $\delta$ and the vector of all twos, and
$\vk$ and the vector of all ones, respectively. The factor $\diri
_{n-1}$ is a density with respect to the $(n-1)$-dimensional vector
$(x_j, j\neq i)$ with corresponding parameters $(\vk_{j}+1, j
\neq i)$. It can also be interpreted as the conditional density of the
$n$-dimensional $\diri_n(x;\vk+1)$, conditioned on $x_i=0$.

Hence, for any $(k_j, j\neq i)$, integers
\begin{eqnarray*}
&&\frac{z_i^{\delta_i+1}}{\Gamma(\delta_i+1)} \prod_{j\neq i} \frac
{z_j^{k_j+\delta_j+1}}{\Gamma(\delta_j + 2 + k_j)} \frac
{x_j^{k_j}}{k_j!}\\[-2pt]
&& \qquad  =\frac{(\delta_i+1)}{\Gamma(\delta_0+N+2n)\Gamma(N+n-1)}\\[-2pt]
&& \qquad  \quad {}\times \diri
_n(z;\vk+\delta+\vtwo) \diri_{n-1}(x; \vk+\vone).
\nonumber
\end{eqnarray*}
Thus \eqref{denform2} reduces to
%
\begin{eqnarray}\label{denform3}
&&(\delta_i+1)  \sum_{N=0}^\infty\frac{\Gamma(\delta_0 + 2n + 2N-1)
B(N+n-1, N+n+\delta_0) }{\Gamma(\delta_0+N+2n)\Gamma(N+n-1)}\nonumber
\\[-9pt]
\\[-9pt]
&& \qquad {}\times\sum_{\vk'\vone=N} \diri_n(z;\vk+\delta+\vtwo) \diri
_{n-1}(x; \vk+\vone).
\nonumber
\end{eqnarray}

However,
\begin{eqnarray*}
&&\frac{\Gamma(\delta_0 + 2n + 2N-1) B(N+n-1, N+n+\delta_0) }{\Gamma
(\delta_0+N+2n)\Gamma(N+n-1)}\\[-2pt]
&& \qquad = \frac{\Gamma(\delta_0 + 2n + 2N-1)}{\Gamma(\delta_0+N+2n)\Gamma
(N+n-1)}\frac{\Gamma(N+n-1)\Gamma(N+n+\delta_0)}{\Gamma
(2N+2n+\delta_0-1)}\\[-2pt]
&& \qquad =\frac{\Gamma(N+n+\delta_0)}{\Gamma(N+2n+\delta_0)}.
\end{eqnarray*}
This completes the proof of formula \eqref{eq:exitnwf}.

The probability in \eqref{eq:exithyp} is a direct consequence of
Theorem~\ref{thm:besexit} conclusion~(i).\vadjust{\goodbreak}
\end{pf*}

\section{Exit time}\label{sec:exittime}

Let $X=(X_1, \ldots, X_n)$ be distributed as $\operatorname{NWF}(-\theta_1/2, \ldots,\break
 -\theta_n/2)$ starting from a point $(x_1, \ldots, x_n)$ in the
unit simplex. Let $\sigma_0$ denote the stopping time
\[
\sigma_0 = \inf\{ t\ge0\dvtx  X_i =0 \mbox{ for some $i$}
\}.
\]
Our objective is to find estimates on the law of $\sigma_0$.

We will simplify the situation by assuming that all $x_i=1/n$ and all
$\theta_i=\theta$. To this end we use the time-change relationship in
Theorem~\ref{bestimechange}. Let $Z=(Z_1, \ldots, Z_n)$ be
independent BESQ processes starting from $(z_1, \ldots, z_n)$ as in
the set-up of Theorem~\ref{bestimechange}, where each $z_i$ is now
one. Then
%
\begin{equation}\label{whatissigma0}
\sigma_0 = 4\int_0^{\tau} \frac{ds}{\zeta(s)}, \qquad\zeta
(s)=\sum_{i=1}^n Z_i(s).
\end{equation}

By Theorem~\ref{thm:besexit}, the distribution of $\tau$ is the same
as considering $n$ i.i.d. $\operatorname{Gamma}(\theta/2 +1)$ random variables $G_1,
\ldots, G_n$, and defining
%
\begin{equation}\label{whatistau}
\tau=\frac{1}{2\max_i G_i}.
\end{equation}

Our first step will be to prove a concentration estimate of $\max_i G_i$.

\begin{lemma}\label{maxmom}
Let $G_1, G_2, \ldots, G_n$ be $n$ i.i.d. Gamma random variables with
parameter $r/2$, for some $r\ge2$. Let $\chi$ be the random variable
$\max_i G_i$.
Then, as $n$ tends to infinity,
\[
E\sqrt{\chi} = \Theta\bigl( \sqrt{\log n} \bigr).
\]
\end{lemma}

\begin{pf} First let $r\in\mathbb{N}$. Let $\{Z_1(i), \ldots,
Z_n(i), i=1,2,\ldots,r \}$ be a collection of i.i.d. standard
Normal random variables. Then $2G_j$ has the same law as $Z_j^2(1)
+\cdots+ Z_j^2(r)$. Hence
\[
E \max_j \abs{Z}_j(1) \le E \sqrt{2\chi} \le\sqrt{r} E \max
_{i,j} \abs{Z}_j(i).
\]
As $n$ tends to infinity, the right-hand side above converges to $\sqrt
{2r\log(rn)}$ while the left-hand side converges to $\sqrt{2\log n}$.
This completes the argument for $r\in\mathbb{N}$. For a general
positive $r$, bound on both sides by $\lfloor r \rfloor$ and $\lfloor
r \rfloor+ 1$.
\end{pf}

We also need a version of logarithmic Sobolev inequality for Gamma
random variables, which can be found in several articles, including
\cite{BW}.

\begin{lemma}[(\cite{BW}, page~2718)]\label{lem:logsobo}
Let $\mu^\theta$ denote the product probability measure of $n$ i.i.d. $\operatorname{Gamma}(\theta)$ random variables. Then, for every $f$ on $\rr^n$
which is in $C^1$ (i.e., once continuously differentiable), one has
%
\begin{equation}\label{logsobo}
\ent(f^2) \le4 \int\Biggl( \sum_{i=1}^n x_i ( \partial_i
f(x) )^2 \Biggr)\,d\mu^{\theta}(x).
\end{equation}
Here $\ent(\cdot)$ refers to the entropy defined by
\[
\ent(f^2)=\int f^2 \log(f^2)\,d\mu^\theta- \biggl( \int f^2\,d\mu
^\theta\biggr) \log\biggl(\int f^2\,d\mu^\theta\biggr).
\]
And $\partial_i$ refers to the partial derivative with respect to the
$i$th coordinate.
\end{lemma}

\begin{lemma}\label{expconcen}
Consider the set-up in Lemma~\ref{lem:logsobo}. Let $F$ be a function
on the open positive quadrant (i.e., every $x_i>0$) which is $C^1$ and satisfies
%
\begin{equation}\label{derivbnd}
\sum_{i=1}^n x_i ( \partial_i F )^2 \le F.
\end{equation}
Then the following concentration estimate holds for any $r >0$:
\begin{eqnarray*}
\mu^\theta\bigl( \sqrt{F} - E_\theta\sqrt{F} \ge r \bigr) &\le
\exp( -r^2 ),\qquad\mu^\theta\bigl( \sqrt{F} -
E_\theta\sqrt{F} \le- r \bigr) \le\exp( -r^2 ),
\end{eqnarray*}
where $E_{\theta}\sqrt F = \int\sqrt{F}\,d\mu^\theta$.
\end{lemma}

\begin{pf} Condition \eqref{derivbnd} implies that $4\sum_{i=1}^n
x_i ( \partial_i \sqrt{F} )^2 \le1$.
Hence, from the classical Herbst argument (e.g., the monograph by
Ledoux~\cite{L}), with a gradient defined by the right-hand side of
\eqref{logsobo}, we get
\[
\mu^\theta\bigl( \sqrt{F} -E_\theta\sqrt{F} > r \bigr) \le\exp
(- r^2 ).
\]
Here $\mu^\theta(\sqrt{F})$ is the expectation of $\sqrt{F}$ under
$\mu^\theta$. Repeating the argument with $-\sqrt{F}$ instead of
$\sqrt{F}$, we get the result.
\end{pf}

%
%
%
%

\begin{theorem}\label{tauconcen}
The random variable $\chi=\max_i G_i$, where $G_i$'s are i.i.d.
$\operatorname{Gamma}(\theta)$ satisfies the following concentration estimate:
%
\begin{equation}\label{eq:chicon}
P \bigl( \sqrt{\chi} > E\bigl(\sqrt{\chi}\bigr) + r \bigr) \le e^{-r^2}
\qquad\mbox{for all } r > 0.
\end{equation}
\end{theorem}

\begin{pf}
To prove \eqref{eq:chicon} we start by noting that Lemma \ref
{expconcen} is satisfied by the family of $\mathbb{L}^k$-norms, $\{
F_{k}, k > 1 \}$, defined by
\[
F_k(x)= \Biggl(\sum_{i=1}^n x_i^k \Biggr)^{1/k}.
\]
This is because each $F_k$ is smooth (when every $x_i$ is positive) and
%
\begin{equation}\label{ineqcheck}
\sum_{i=1}^n x_i (\partial_i F_{k}(x) )^2  = \sum
_{i=1}^n x_i \biggl[ \frac{x_i^{k-1}}{ (\sum_{j=1}^n x_j^k
)^{1-1/k}} \biggr]^2=\frac{\sum_{i=1}^n x_i^{2k-1}}{ (
\sum_{j=1}^n x_j^k )^{2-2/k}}.
\end{equation}
Since, for any nonnegative $y_1, y_2, \ldots, y_n$ and any $\beta
>1$, one has
\[
\sum_{i=1}^n y_i^{\beta} \le\Biggl( \sum_{i=1}^n y_i \Biggr)^\beta,
\]
applying it for $y_i= x_i^k$ and $\beta= 2-1/k$, we get
\[
\sum_{i=1}^n x_i^{2k-1} \le\Biggl( \sum_{i=1}^n x_i^{k} \Biggr)^{2-1/k}.
\]
Combining the above with \eqref{ineqcheck}, we get
\[
\sum_{i=1}^n x_i (\partial_i F_{k}(x) )^2 \le\Biggl(\sum
_{i=1}^n x_i^k \Biggr)^{1/k}=F_k(x).
\]
Thus $F_{k}$ satisfies condition \eqref{derivbnd}.

Since $F_k$ converges pointwise to $\max_i x_i$ as $k$ tends to
infinity, by applying DCT, Lemma~\ref{expconcen} is true for the
function $\max_i G_i$. This proves \eqref{eq:chicon}.
\end{pf}

%

Our next step will be to prove estimate on the quantity $\sigma_0$ in
\eqref{whatissigma0}. The process $\zeta(s)$ is non-Markovian and not
distributed as $Q^{-n\theta}$. However, on an possibly enlarged sample
space, one can create a $Q^{-n\theta}$ process $\tilde\zeta$, such
that the paths of $\zeta$ and $\tilde\zeta$ are indistinguishable
until $\sigma_0$. This is possible by considering the SDE solved by~$\zeta$,
\[
\zeta(t)= n - n \theta t + \int_0^t \sqrt{\zeta(s)}\,dW(s), \qquad t
< \sigma_0.
\]
To extend the process beyond $\sigma_0$, one concatenates an
independent Brownian motion $\widetilde W$ and defines
\[
\beta(t) =
\cases{\displaystyle
W(t) ,&\quad $t \le\sigma_0$,\cr\displaystyle
W(t) + \widetilde W(t-\sigma_0) ,&\quad $t > \sigma_0$.
}
\]
Then $\beta$ is a Brownian motion in the enlarged filtration. Since
$Q^{-n\theta}$ admits a strong solution, the process
%
\begin{equation}\label{extendzeta}
\tilde\zeta(t) = n - n \theta t + 2\int_0^t \sqrt{\tilde\zeta(s)}\,d\widetilde W(s),\qquad t < T_0,
\end{equation}
has law $Q^{-n\theta}$ and pathwise indistinguishable from $\zeta$
until time $\sigma_0$. Thus in the following discussion we will treat
as if $\zeta$ itself is distributed as $Q^{-n\theta}$, keeping in
mind the above construction.

\begin{theorem}\label{thm:exittime}
Let $\mu$ be distributed as an $n$-dimensional $\operatorname{NWF}(\delta, \delta,
\ldots, \delta)$ starting from the point $(1/n, 1/n,\ldots, 1/n)$.
Let $\sigma_0$ be the first time that any of the coordinates of $\mu$
hit zero. Let
\[
a_n = E \max_{1\le i \le n} \sqrt{G_i}, \qquad G_i \stackrel{\mathrm{i.i.d.}}{\sim}  \operatorname{Gamma}(\delta+1).
\]
Then, $a_n = \Theta(\sqrt{\log n})$, $\sigma_0$ has the law given by
\eqref{whatissigma0} where $\zeta$ is distributed as $Q^{-2n\delta
}_1$, and $\tau$ is a random time.

Moreover, for any $r > 0$, we get
\begin{eqnarray*}
P \biggl( \frac{1}{n (a_n + r)} \le\sqrt{2\tau} \le\frac
{1}{n(a_n+r)} \biggr) \ge1 - 2 e^{-r^2}.
\end{eqnarray*}
\end{theorem}

\begin{rmk}
It is impossible to provide a simple description of the exact
distribution of $\sigma_0$, due to the distributional dependence of
$\zeta$ and $\tau$. The above theorem shows that $\tau$ is about a
constant, and one can compare the distribution of $\sigma_0$ with that
of $4\int_0^{\cdot} du/\zeta(u)$, where the upper limit of the
integral is a constant. Limiting large deviation behavior of such
integrals, it is possible to derive by methods as in~\cite{YZ}.
\end{rmk}

\begin{pf*}{Proof of Theorem~\ref{thm:exittime}} The proof is
obvious from Lemma~\ref{tauconcen} and expression \eqref{whatistau}.
\end{pf*}

\begin{appendix}\label{appm}
\section*{Appendix: Proofs of properties of BESQ processes}
\begin{pf*}{Proof of Lemma~\ref{scale}} We use Exercise 3.20 in
\cite{RY}, page~311. The scale function for $Q^{\theta}$ for $\theta
\ge0$ is well known to be $x^{-\theta/2+1}$ (see~\cite{RY}, page~443). Nearly identical calculations lead to the case when
$\theta$ is replaced by $-\theta$, and we obtain the scale function
$s(x)= x^{\theta/2+1}$.

The speed measure is the measure with the density
\[
m'(x)=\frac{2}{s'(x)4x}=\frac{1}{2(\theta/2+1)}x^{-\theta/2-1}.
\]
We now use Feller's criterion to check if the origin is an entry and/or
exit point (see~\cite{itomckean}, page~108). Note that
%
\begin{eqnarray}\qquad
m(\xi,1/2)&=&\frac{1}{2(\theta/2+1)}\int_{\xi}^{1/2}x^{-\theta
/2-1}\,dx=\frac{1}{\theta(\theta/2+1)} ( \xi^{-\theta/2} -
2^{\theta/2} ),\nonumber
\\[-8pt]
\\[-8pt]
m(0,\xi]&=&\infty  \qquad\mbox{for all positive } \xi.
\nonumber
\end{eqnarray}
Thus
\[
\int_0^{1/2} m(\xi,1/2] s(d\xi) < \infty  \quad\mbox{and}\quad
\int_0^{1/2} m(0,\xi] s(d\xi)=\infty.
\]
This proves that the origin is an exit and not an entry.

Finally, to obtain part (ii) we apply Girsanov's theorem~\cite{RY}, page~327. Let $X$ satisfy the SDE $d X(t)= -\theta \,dt + 2\sqrt
{X(t)}\,d\beta(t)$; then we take $D(t)=X^{\theta/2+1}(t)$ (without the
normalization, for simplicity) and apply Girsanov. Under the changed
measure, there is a standard Brownian motion $\beta^*$, such that
\begin{eqnarray*}
\beta(t) &=& \beta^*(t) + \int_0^t X^{-\theta/2-1}(s)\,d\langle
\beta , D \rangle_s\\
&=&\beta^*(t) + \int_0^t X^{-\theta/2-1}(s) ( \theta+ 2
) X^{\theta/2+1/2}(s)\,ds\\
&=&\beta^*(t) + ( \theta+ 2 )
\int_0^t X^{-1/2}(s)\,ds.
\nonumber
\end{eqnarray*}

Thus under the changed measure,
\begin{eqnarray*}
dX(t) &=& -\theta\, dt + 2 X^{1/2}(t)\,d\beta(t) = -\theta\, dt + 2(\theta
+2)\,dt + d\beta^*(t) \\
&=& (\theta+4)\,dt + d \beta^*(t).
\end{eqnarray*}
The interpretation as the conditional distribution is classical (see
\cite{besselbridge}).
\end{pf*}

\begin{pf*}{Proof of Lemma~\ref{logct}}
For the assertion it is enough to take $t=1$. Note that, under
$Q_0^{\theta}$, the coordinate process satisfies time-inversion; that
is, the process $\{ t^2 Z(1/t), t\ge0\}$ has law $Q_0^\theta$.
Thus, for $0<\epsilon< 1$, if we define
\[
U_{\epsilon}= \int_{\epsilon}^1 \frac{du}{Z(u)}= \int
_{1}^{1/\epsilon} \frac{dt}{t^2 Z(1/t)},
\]
then $U_\epsilon$ has the same law as $C_{1/\epsilon} - C_1 = \int
_{1}^{1/\epsilon}\,du/Z(u)$.
Thus, by~\cite{YZ}, Theorem~1.1, we get $\lim_{\epsilon\rightarrow
0} U_{\epsilon}/\log(1/\epsilon) = (\theta- 2)^{-1}$ almost surely.
\end{pf*}
\end{appendix}

\section*{Acknowledgments}
I thank David Aldous, Zhen-Qing Chen, Michel Le\-doux and Jon Wellner for
very useful discussions.
I thank the anonymous referee for a thorough review which led to a
significant improvement of the article.


%

\printaddresses

\end{document}